\documentclass[11pt]{article}
\input{preamble.tex}

\bibliography{refs}

\title{Dirichlet Functors are Contravariant Polynomial Functors}
\author{David Jaz Myers \and David I.\ Spivak}
\date{April 2020}

\begin{document}

\maketitle

\begin{abstract}
  Polynomial functors are sums of covariant representable functors from the
  category of sets to itself. They have a robust theory with many applications
  --- from operads and opetopes to combinatorial species. In this paper, we
  define a contravariant analogue of polynomial functors: Dirichlet functors. We
  develop the basic theory of Dirichlet functors, and relate them to their
  covariant analogues.
\end{abstract}

\section{Introduction}
A polynomial functor $P : \Set \to \Set$ is a sum of representables
\begin{equation}\label{eqn.poly}
P(X) := \sum_{b \in B} X^{E_b}
\end{equation}
where the family of sets $E_b$ depends on $b\in B$. This data is known
in the computer science literature as a ``container'' \cites{abbott2003categories}{abbott2005containers}{abbot2003categoriesthesis}, but since such an indexed
family of sets can be represented by as a function
\[
  \begin{tikzcd}
    E \arrow[d, "\pi"] \\
    B
  \end{tikzcd}
\]
we will refer to it as a \emph{bundle}.

Remarkably, all natural transformations between polynomial functors can be
represented in terms of their associated bundles. A natural transformation $P
\to P'$ corresponds to a dependent pair of functions
\[
\left( f : B \to B',\, f^{\sharp} : \dprod{b : B} E'_{fb} \to E_b \right)
\]
noting that $f$ acts covariantly on the base and $f^\sharp$ acts contravariantly on
fibers. 
In terms of bundles, a map $\pi\to\pi'$ is a diagram of the following sort:
\[
  \begin{tikzcd}
    E \arrow[d, "\pi"'] \arrow[r, leftarrow, "f_{\sharp}"] & \bullet \arrow[d] \arrow[r]
    \arrow[dr, phantom, "\ulcorner" very near start] & E' \arrow[d, "\pi'"] \\
    B \arrow[r, equals] & B \arrow[r, "f"'] & B'
  \end{tikzcd}
\]
We refer to these special spans as \emph{contravariant} morphisms of bundles.

\begin{thm}[Theorem 2.17 of \cite{kock2012polynomial}]
The category of polynomial functors and natural transformations is equivalent to
the category of bundles and contravariant morphisms.
\end{thm}

This begs the question: what, then, are we to make of the more obvious,
\emph{covariant} morphisms of bundles
\[
  \begin{tikzcd}[column sep=large]
E \arrow[r, "\term{tot}(f_{\sharp})"] \arrow[d, "\pi"'] & E' \arrow[d, "\pi'"] \\
B \arrow[r, "f"']                                       & B'                  
\end{tikzcd}
\]
for which $f_\sharp$ is covariant
in each fibers?

It turns out that these covariant maps of bundles correspond to natural
transformations between the appropriate sums of \emph{contravariant}
representables:
$$D(X) := \sum_{b \in B} E_b^X.$$

Polynomial functors as in \eqref{eqn.poly} get their name from the case in which
$B$ and $E$ are finite sets. Consider the function $\term{card}~E_{(-)} : B \to
\Nb$, which takes the cardinality of each fiber $E_b$. Letting $B_n :=
(\term{card}E_{(-)})\inv(n)$, i.e.\ $B_n$ is the set of elements whose fibers have $n$ elements,
we find that the for any set $X$, the cardinality
$$|P(X)| = \sum_{n \in \Nb} |B_n||X|^n$$
is a polynomial in the cardinality of $X$. Similarly,
$$|D(X)| = \sum_{n \in \Nb} |B_n|n^{|X|}$$
resembles a Dirichlet series in the cardinality of $X$ -- without the usual
negative sign.\footnote{Future work of the first author will recover the
  negative sign by generalizing the theory of Dirichlet functors to homotopy
  types.} Accordingly, we call such
sums of contravariant representables \emph{Dirichlet functors}.

We will show in this paper that Dirichlet functors are, quite robustly, the
contravariant analogue of polynomial functors. In particular, the many equivalent ways to
say that a functor is polynomial have contravariant analogues.
\begin{thm}[See discussion in 1.18 of \cite{kock2012polynomial}]\label{thm:polynomial.set.characterization}
  Let $P : \Set \to \Set$ be a functor. Then the following are equivalent.
  \begin{enumerate}
  \item $P$ is polynomial.
  \item $P$ is the sum of covariant representables.
  \item There is a bundle $\pi : E \to B$ together with a natural isomorphism
    $$P(X) \cong \sum_{b \in B} X^{E_b}.$$
    Or, equivalently, a natural isomorphism of $P$ with the composite
    $$\Set \xto{{\Delta_{!_E}}} \Set_{/E} \xto{{\Pi_\pi}} \Set_{/B} \xto{{\Sigma_{!_B}}} \Set.$$
  \item $P$ is accessible and preserves connected limits.
  \end{enumerate}
\end{thm}

Analogously, we will prove the following theorem.
\begin{thm}
Let $D : \Set\op \to \Set$ be a contravariant functor. Then the following are
equivalent.
\begin{enumerate}
\item $D$ is Dirichlet.
\item $D$ is the sum of contravariant representables.
\item There is a bundle $\pi : E \to B$ together with a natural isomorphism
  $$D(X) \cong \sum_{b \in B} E_b^X.$$
  Or, equivalently, a natural isomorphism of $D$ with the composite
  $$\Set\op \xto{(\Delta_{!_B})\op} (\Set_{/B})\op \xto{\Set_{/B}(-, \pi)} \Set_{/B}
  \xto{\Sigma_{!_B}} \Set.$$
\item $D$ preserves connected limits.
\end{enumerate}
\end{thm}

Note that we no longer need to assume accessiblity. This is a general feature of
the theory of Dirichlet functors; it is a bit ``smaller'' and more manageable
than that of polynomials. In particular, a Dirichlet functor is determined by
its action on the terminal morphism $!_0 : 0 \to 1$ of the empty set. As a corollary, Dirichlet functors form a topos.

\begin{thm}\label{thm:dirichlet.set.equivalence}
The functor $\Set^{\down} \to \type{Fun}(\Set\op, \Set)$, given by sending $\pi : E
\to B$ to the induced Dirichlet functor $X \mapsto \sum_{b \in B} E_b^X$, is
fully faithful, and so gives an equivalence
$$\Set^{\down} \simeq \Dir$$
between the topos of bundles and the category of Dirichlet functors, with inverse given by evalutation at
$!_0 : 0 \to 1$.
\end{thm}

Now, object-wise, a Dirichlet functor and a polynomial functor are determined by
the same data --- a bundle $\pi : E \to B$ of sets. Accordingly, one would
expect for any set $N$ a transformation
$$X^N \mapsto N^X$$
turning polynomial functors into Dirichlet functors, and vice versa. But the
natural transformations between each sort of functor induce different morphisms
between the bundles; natural transformations between polynomial functors induce
contravariant bundle morphisms, while natural transformations between Dirichlet
functors induce covariant bundle morphisms. However, if we restrict to those
morphisms of bundles which are \emph{isovariant} on the fibers --- that is, the
pullback diagrams of the form
\[
  \begin{tikzcd}[column sep=large]
E \arrow[r, "\term{tot}(f_{\sharp})"] \arrow[d, "\pi"'] \arrow[dr, phantom,
"\ulcorner" very near start] & E' \arrow[d, "\pi'"] \\
B \arrow[r, "f"']                                       & B'                  
\end{tikzcd}
\]
which preserve the number of elements in each row --- we find that such a morphism
is both a co- and contravariant morphism of bundles. It is well known that such
\emph{cartesian} morphisms of bundles correspond to \emph{cartesian} natural
transformations between polynomial functors \cite[Theorem 3.8]{kock2012polynomial} --- those whose naturality
squares are pullbacks. This is true as well for Dirichlet functors.
\begin{thm}
A natural transformation $D \to D'$ of Dirichlet functors is Cartesian if and
only if the corresponding bundle map
\[
  \begin{tikzcd}[column sep=large]
E \arrow[r, "\term{tot}(f_{\sharp})"] \arrow[d, "\pi"'] \arrow[dr, phantom,
"\ulcorner" very near start] & E' \arrow[d, "\pi'"] \\
B \arrow[r, "f"']                                       & B'                  
\end{tikzcd}
\]
is a pullback. As a corollary, we have an equivalence of categories
$$\type{Poly}_{\ulcorner} \simeq \Dir_{\ulcorner}$$
between polynomial functors with cartesian natural transfromations and Dirichlet
functors with cartesian natural transfromations.
\end{thm}

\begin{acknowledgements}
We appreciate helpful comments from Andr\'e Joyal and Joachim Kock on an early draft of this paper. Myers was supported by the National Science Foundation grant DMS-1652600, and Spivak was supported by AFOSR grants FA9550-19-1-0113 and FA9550-17-1-0058.
\end{acknowledgements}

\section{Dirichlet Functors} \label{sec:set.level}

Before diving in to the theory of Dirichlet functors, let's first consider the
category $\Set^{\down}$ of bundles of sets and covariant bundle maps. For our
proofs to go smoothly, we will need to explicitly keep track of the
self-dualizing isomorphism $\downarrow \xto{\sim} (\downarrow)\op$ on the walking arrow.

\begin{defn}
  We let $\down$ be the walking arrow --- the category $\term{dom} \to
  \term{cod}$ consisting of a single morphism. We denote by $\sigma : \down
  \to (\downarrow)\op$ the self-dualizing isomorphism of $\down$, and note
  that $\sigma\inv = \sigma\op$.
\end{defn}

\begin{prop}\label{prop:adjoint.sextuple}
  There is an adjoint sextuple:
  \[
    \begin{tikzcd}
      \Set \arrow[r, "\const"] \arrow[r, leftarrow, shift left = 5,
      "\term{cod}"] \arrow[r, leftarrow, shift right = 5, "\term{dom}"]
      \arrow[r, shift left = 10, "!^{(-)}"]  \arrow[r, shift right = 10,
      "!_{(-)}"] \arrow[r, leftarrow, shift left = 15, "\term{ZC}"]& \Set^{\down}
    \end{tikzcd}
  \]
  All three functors $\Set \to \Set^{\down}$ are fully faithful.
\end{prop}
\begin{proof}
There is an adjoint triple $\begin{tikzcd} 1 \arrow[r,leftarrow] \arrow[r,
  bend left] \arrow[r, bend right] & \down \end{tikzcd}$. The middle three
functors---$\term{cod}$, $\const$, and $\term{dom}$---are given by restricting along this adjoint triple. The next two, $!_{(-)}$ and $!^{(-)}$, are
given by Kan extending, or more concretely:
  \begin{equation}\label{eqn.bangbang}
  X \mapsto \begin{tikzcd} 0 \arrow[d, "!^X"] \\ X \end{tikzcd}
  \qquad\text{and}\qquad
  X \mapsto \begin{tikzcd} X \arrow[d, "!_X"] \\ 1 \end{tikzcd}
  \end{equation}
 It is easy to see that the unit $X \to \term{dom} !_{X}$ is an isomorphism (it is in
fact an identity), so $!_{(-)}$ is fully faithful. Therefore, all three functors
going from $\Set \to \Set^{\down}$ are fully faithful.

The existence of the final adjoint can be deduced from the fact that $!^{(-)}$ preserves all
limits. It sends a bundle $\pi : E \to B$ to the largest subset $\term{ZC}(\pi)$ of $B$ for
which the following square is a pullback
\[
  \begin{tikzcd}[column sep=large]
0 \arrow[r] \arrow[d] \arrow[dr, phantom,
"\ulcorner" very near start] & E \arrow[d, "\pi"] \\
\term{ZC}(\pi) \arrow[r, hookrightarrow]                                       & B
\end{tikzcd}
\qedhere
\]
\end{proof}
We prove a few quick facts we will use later.
\begin{lem}\label{lem:bang.preserves.connected.colimits}
  The functor $!_{(-)}\colon\Set\to\Set^{\down}$ from \eqref{eqn.bangbang} preserves connected colimits.
\end{lem}
\begin{proof}
  To see that $!_{(-)}$ preserves connected colimits, recall that colimits in $\Set^{\down}$ are calculated pointwise. It remains to show, then, that the map of bundles
  \[
    \begin{tikzcd}
    \colim X_i \arrow[r, equals] \arrow[d] & \colim X_i \arrow[d] \\
    \colim 1 \arrow[r] & 1 
    \end{tikzcd}
  \]
  is an isomorphism in $\Set^{\down}$, for which it suffices to show that
  $\colim 1$ is terminal. But the colimit of a diagram of terminal objects is the set
  of connected components of its indexing category. Since we assumed the
  indexing category was connected, this contains a single element.
\end{proof}

\begin{lem}\label{lem:set.bundle.yoneda}
The Yoneda embedding $\yo : (\downarrow)\op \to \Set^{\downarrow}$ is equal to the
composite $$(\downarrow)\op \xto{\sigma\op} \,\downarrow\, \xto{!_{0}} \Set
\xto{!_{(-)}} \Set^{\downarrow}.$$

As a corollary, any $\pi : \downarrow \to \Set$ is naturally isomorphic to the composite
$$\downarrow \xto{\sigma} (\downarrow)\op \xto{!_0\op} \Set\op \xto{!_{(-)}\op}
(\Set^{\downarrow})\op \xto{\Set^{\downarrow}(-, \pi)} \Set$$
by the Yoneda lemma.
\end{lem}
\begin{proof}
One checks directly.
\end{proof}

Now we will define the \emph{extent} of a bundle $\pi$ to be the Dirichlet functor $\ext_\pi: \Set\op\to\Set$ that it
corresponds to. This sends a bundle
$\pi : E \to B$ to the functor
$$\ext_\pi(X)\coloneqq \sum_{b \in B} E_b^X.$$
We will, however, give a more abstract definition of the extent, and then
calculate a number of presentations of it.

\begin{defn}\label{def.extent}
  Consider the functor $!_0\op \circ \sigma : \down \to \Set\op$ picking out the
  unique morphism $1\to 0$ in $\Set\op$. Sending a functor $\Set\op\to\Set$ to
  its precomposition with $!_0\op \circ\sigma$ gives an evaluation functor
  $$\Fun(\Set\op, \Set)\xto{\ev_{!_0}} \Set^{\down}.$$
  This functor admits a right adjoint by right Kan extension along $!_0\op \circ
  \sigma : \down \to \Set\op$; we define the
  \emph{Dirichlet extent} functor $\ext : \Set^{\down} \to \Fun(\Set\op,
  \Set)$ to be this right adjoint. It sends any bundle $\pi$ to
  $$\ext_{\pi} :\equiv \term{ran}_{!_0\op \circ \sigma} \pi.$$
\end{defn}

\begin{prop}\label{prop:set.characterizing.extent}
  Let $\pi : E \to B$ be a bundle. The following are equivalent:
  \begin{enumerate}
  \item The extent $\ext_{\pi} : \Set\op \to \Set$ of $\pi$ from \cref{def.extent}.
  \item The functor
    $$X \mapsto \sum_{b \in B} E_b^X,$$
    or equivalently the composite
  $$\Set\op \xto{(\Delta_B)\op} (\Set_{/B})\op \xto{\Set_{/B}(-, \pi)} \Set_{/B}
  \xto{\Sigma_B} \Set.$$
\item The pullback in $\Fun(\Set\op, \Set)$:  
\[
\begin{tikzcd}
	\ext_\pi \ar[r]\ar[d]\arrow[dr, phantom,
      "\ulcorner" very near start]&
	B\ar[d]\\
	\Set(-,E)\ar[r]&
	\Set(-,B)
\end{tikzcd}
\]
\item The restricted representable functor
  $\Set^{\down}(!_{(-)}, \pi)$.
\item The functor
  $$X \mapsto \lim(\Hom_{\Set}(!_0, X)\op \to \down\op \xto{\sigma\op} \down \xto{\pi} \Set)$$
  where $\Hom_{\Set}(!_0, X)$ is the comma category of $!_0 : \down \to
  \Set$ over $X:\ast\to\Set$.
\item The functor
  $$X \mapsto \lim(( X^{\triangleleft} )\op \xto{( !^{\triangleleft} )\op}
  (1^{\triangleleft})\op \xto{\sigma\op} \down
  \xto{\pi} \Set)$$
  where $(-)^{\triangleleft}$ is the \emph{left cone} 2-functor, adjoining an initial object.
  \end{enumerate}
\end{prop}
\begin{proof}
We have presented these results in order of most understandable to most
computational; we will prove it a somewhat opposite order.

First, we note that the conical limit formula for $\ext_{\pi} \equiv
\term{ran}_{!_0} \pi$ as a right Kan extension says
$$\ext_\pi(X)= \lim(\Hom_{\Set\op}(X, !_0\op \circ \sigma) \to \down \xto{\pi} \Set).$$
Now,
$\Hom_{\Set\op}(X, !_0\op \circ \sigma) \simeq \Hom_{\Set}(!_0 \circ \sigma\op,
X)\op$ over $\down$. Furthermore, we have the following equivalence:
\[
\begin{tikzcd}
{\Hom_{\Set}(!_0 \circ \sigma\op, X)\op} \arrow[d, "\sim"'] \arrow[r] & \down \arrow[d, "\sigma"] \\
{\Hom_{\Set}(!_0, X)\op} \arrow[r]                                    & \down\op                 
\end{tikzcd}
\]
and therefore we may equivalently calculate this limit as 
$$\lim(\Hom_{\Set}(!_0, X)\op \to \down\op \xto{\sigma\op} \down \xto{\pi} \Set).$$
This gives us the equivalence between (1) and (5).

The comma category $\Hom_{\Set}(!_0, X)$ simply adjoins an initial object
to (the discrete category) $X$. Therefore, we find that (5) and (6) are equivalent.

Every set $X$ is the colimit of the diagram $X^{\triangleleft}
\xto{1^{\triangleleft}}\down \xto{!_0} \Set$, namely:
\[
\begin{tikzcd}
              &              & X                                               &              &               \\
1 \arrow[rru] & 1 \arrow[ru] & \cdots                                          & 1 \arrow[lu] & 1 \arrow[llu] \\
              &              & 0 \arrow[llu] \arrow[lu] \arrow[ru] \arrow[rru] &              &              
\end{tikzcd}
\]
Since, by Lemma \ref{lem:bang.preserves.connected.colimits}, $!_{(-)}$
preserves connected colimits, we may make the following identification of (4)
with (6) using Lemma \ref{lem:set.bundle.yoneda}: 
\begin{align*}
  \Set^{\down}(!_X, \pi) &= \Set^{\down}(!_{\colim(X^{\triangleleft} \xto{!^{\triangleleft}} 1^{\triangleleft} \xto{!_0} \Set)}, \pi) \\
  &\simeq \Set^{\down}(\colim(X^{\triangleleft} \xto{!^{\triangleleft}} 1^{\triangleleft} \xto{!_0} \Set \xto{!_{(-)}} \Set^{\down}), \pi)\\
&\simeq \lim((X^{\triangleleft})\op \xto{(!^{\triangleleft})\op} (1^{\triangleleft})\op \xto{!_0\op} \Set\op \xto{!_{(-)}\op} (\Set^{\down})\op \xto{\Set^{\down}(-, \pi)} \Set) \\
&\simeq \lim((X^{\triangleleft})\op \xto{(!^{\triangleleft})\op} (1^{\triangleleft})\op \xto{\sigma\op} \down \xto{\pi} \Set).
\end{align*}

We see that (3) is equivalent to (4) by noting that the following square of
natural transformations is a pullback:
\[
\begin{tikzcd}
{\Set^{\downarrow}(-, -)} \arrow[d] \arrow[r] \arrow[rd, phantom, "\ulcorner" very near start] & {\Set(\cod(-), \cod(-))} \arrow[d] \\
{\Set(\dom(-), \dom(-))} \arrow[r]                       & {\Set(\dom(-), \cod(-))}          
\end{tikzcd}
\]
and restricting the right side to $\pi$ and the left side to $!_{(-)}$.

Finally, we note that the set $\Set^{\down}(!_X, \pi)$ is naturally
isomorphic to the set $\sum_{b \in B} E_b^X$, letting us identify (4) with (2).
\end{proof}

Now we are ready to intrinsically characterize the Dirichlet functors.
\begin{defn}
A \emph{Dirichlet functor} is a contravariant functor $D : \Set\op \to \Set$
which preserves connected limits. We denote by $\Dir$ the category of Dirichlet functors and natural transformations.
\end{defn}

\begin{thm}\label{thm:dirichlet.set.characterization}
  The functor $\ext$ is fully faithful, and gives an equivalence
  $$\Set^{\down} \simeq \Dir.$$

  As a corollary, the category of Dirichlet functors is a topos.
\end{thm}
\begin{proof}
Since $!_0 : \down \to \Set\op$ is fully faithful, the counit of the
$\ev_{!_0} \vdash \ext$ adjunction, the universal cell defining the
right Kan extension $\ext_-$, is an isomorphism. Thus $\ext$ is fully faithful. In what follows we show that a functor is Dirichlet---preserves connected limits---if and only if it is the extent of a bundle, proving the equivalence of $\Dir$ with $\Set^{\down}$.

If $D$ is
the extent of a bundle $\pi$, then by Prop \ref{prop:set.characterizing.extent}, $D$ is naturally
isomorphic to the restricted representable $\Set^{\down}(!_{(-)},\pi)$. By Lemma
\ref{lem:bang.preserves.connected.colimits}, this sends connected colimits in $\Set$ to
connected limits.

Now, we show that if $D$ is Dirichlet, then the unit
$D \to \ext_{D!_0}$ is an isomorphism. By Prop
\ref{prop:set.characterizing.extent}, 
  $$\ext_{D!_0}(X) = \lim(X^{\triangleleft} \xto{!^{\triangleleft}} 1^{\triangleleft}
  \xto{D!_0} \Set).$$
  Every set $X$ is the connected colimit of the diagram $X^{\triangleleft} \to
\down \xto{!_0} \Set$, and therefore if $D$ preserves this limit, then
$D(X)$ is precisely the above limit $\ext_{D!_0}(X)$.
\end{proof}

\begin{rmk}
 Since polynomial functors preserve connected limits, the composite $P \circ D$
 of a polynomial functor after a Dirichlet functor is Dirichlet. On the other
 hand, the composite $D' \circ D\op$ of two Dirichlet functors does not in
 general preserve connected limits, since $D\op$ sends connected colimits in
 $\Set$ to connected colimits in $\Set\op$, and $D'$ does not necessarily
 preserve these. Furthermore, composites of Dirichlet functors are not in
 general accessible.   
\end{rmk}

\begin{rmk}
The six adjoints of \cref{prop:adjoint.sextuple} correspond, under
the equivalence of \cref{thm:dirichlet.set.characterization}, to
\begin{align*}
  \term{ZC}(D!_0) &\cong \mbox{the coefficient of $0^X$ in $D$.}\\
  \ext(!^{C}) &\cong X \mapsto C \times 0^X \\
  \cod(D!_0) &\cong D(0) \\
  \ext( \const(C) ) &\cong X \mapsto C \\
  \dom(D!_0) &\cong D(1) \\
  \ext( !_{C} ) &\cong X \mapsto C^X 
\end{align*}
 In particular, $!_{(-)}$ corresponds to the Yoneda embedding.
\end{rmk}

\section{Cartesian Transformations between Dirichlet Functors}
In this section, we turn to cartesian transformations between Dirichlet
functors. We will show that the category of Dirichlet functors and cartesian
transfromations is equivalent to the category of polynomial functors and
cartesian transformations.

\begin{prop}\label{prop:set.dirichlet.cartesian.iff.bundle.cartesian}
A natural transformation $\phi : D \to D'$ between Dirichlet functors is
cartesian if and only if the induced bundle map $D!_0 \to D'!_0$ is a pullback.

As a corollary, the equivalence $\Dir \simeq \Set^{\down}$ restricts to an
equivalence
$$\Dir_{\ulcorner} \simeq \Set^{\down}_{\ulcorner}$$
between Dirichlet functors with cartesian natural transformations and bundles with
pullback squares.
\end{prop}
\begin{proof}
We want to show that for any $f\colon D\to D'$, the square
\begin{equation}\label{eqn.cart1}
\begin{tikzcd}
	D(1)\ar[r, "f_1"]\ar[d, "\pi"']&
	D'(1)\ar[d, "\pi'"]\\
	D(0)\ar[r, "f_0"']&
	D'(0)
\end{tikzcd}
\end{equation}
is a pullback in $\Set$ iff for all functions $g\colon X\to X'$, the naturality square
\begin{equation}\label{eqn.cart2}
\begin{tikzcd}
  D(X')\ar[r, "f_{X'}"]\ar[d, "D(g)"']&
  D'(X')\ar[d, "D'(g)"]\\
  D(X)\ar[r, "f_X"']&
  D'(X)
\end{tikzcd}
\end{equation}
is a pullback in $\Set$. We will freely use the natural isomorphism
$D(X)\cong\Set^{\down}(!_X, D!_0)$ from
\cref{prop:set.characterizing.extent}, which allows us to identify Diagram
\eqref{eqn.cart2} with
\begin{equation}\label{eqn.cart3}
\begin{tikzcd}
 \Set^{\down}(!_{X'}, D!_0)     \ar[r, "f_{!_0,1}"]\ar[d, "!_g^{1}"']& \Set^{\down}(!_{X'}, D'!_0)
  \ar[d, "!_{g}^{1}"]\\
  \Set^{\down}(!_{X}, D!_0)\ar[r, "f_{!_0, 1}"']&
 \Set^{\down}(!_X, D'!_0) 
\end{tikzcd}
\end{equation}

The square in Diagram \eqref{eqn.cart1} is a
special case of that in Diagram \eqref{eqn.cart2}, namely for $g\coloneqq !_0$;
this establishes the only-if direction. 

To complete the proof, suppose that Diagram \eqref{eqn.cart1} is a pullback, take an arbitrary $g\colon X\to X'$, and suppose given a commutative solid-arrow diagram as shown:
\[
\begin{tikzcd}[sep=small]
  X\ar[rr, "g"]\ar[dd]\ar[rd]&&
  X'\ar[dr]\ar[dd]\ar[dl, dotted]\\&
  D(1)\ar[rr, crossing over]&&
  D'(1)\ar[dd]\\
  1\ar[dr]\ar[rr, equal]&&
  1\ar[dr]\ar[dl, dotted]\\&
  D(0)\ar[from=uu, crossing over]\ar[rr]&&
  D'(0)
\end{tikzcd}
\]
We can interpret the statement that Diagram \eqref{eqn.cart3} is a pullback as
saying that there are unique dotted arrows making the diagram commute. So, we need to show that if the front face is a pullback, then there are unique diagonal dotted arrows as shown, making the diagram commute. This follows quickly from the universal property of the pullback.
\end{proof}

\begin{thm}\label{thm:set.poly.cart.equiv.dir.cart}
  There is an equivalence of categories
  $$\type{Poly}_{\ulcorner} \simeq \Dir_{\ulcorner}$$
  between the category of polynomial functors and cartesian transformations and
  Dirichlet functors and cartesian transformations. This equivlalence sends
  representables to representables
  $$(-)^N \mapsto N^{(-)}.$$
\end{thm}
\begin{proof}
This follows by composing the equivalence of $\Dir_{\ulcorner}$ with
$\Set^{\down}$ from
\cref{prop:set.dirichlet.cartesian.iff.bundle.cartesian} with that of
\cite[Proposition 3.14]{kock2012polynomial}, noting that
$\Set^{\down}_{\ulcorner}$ is the category of $(1,1)$-polynomials. 
\end{proof}

\begin{cor}
  Let $D$ be a Dirichlet functor. Then the category
  $\Dir_{\ulcorner/D}$
  of Dirichlet functors with a cartesian map to $D$ is a topos.
\end{cor}
\begin{proof}
By Theorem \ref{thm:set.poly.cart.equiv.dir.cart}, this category is equivalent
to $\type{Poly}_{\ulcorner/P}$ for a polynomial $P$. But this is a topos as
observed in \cite[Remark 2.6.2]{GHK:Analytic.Monads}.
\end{proof}

Now, since $\Dir$ is a topos, so is $\Dir_{/D}$. And, as we saw above,
$\Dir_{\ulcorner/D}$ is a topos as well. What is the relationship between
$\Dir_{/D}$ and $\Dir_{\ulcorner/D}$?

We will show that $\Dir_{/_{\ulcorner} D}$ is a subtopos of $\Dir_{/D}$ with the
left exact left adjoint to the inclusion given by the vertical / cartesian factorization system on $\Set^{\down}$.

\begin{thm}
  For any $\pi : \down \to \Set$, we have a subtopos inclusion
  $$\Set^{\down}_{/_{\ulcorner} \pi} \hookrightarrow \Set^{\down}_{/\pi}$$

  with left exact left adjoint given by the vertical / cartesian factorization system:
  \[
\begin{tikzcd}
E' \arrow[r] \arrow[d] & E \arrow[d, "\pi"] \\
B' \arrow[r]           & B                 
\end{tikzcd} \quad\quad\mapsto\quad\quad
\begin{tikzcd}
E' \arrow[d] \arrow[r, dashed]    & \bullet \arrow[d] \arrow[r] \arrow[rd, phantom, "\ulcorner" very near start] & E \arrow[d, "\pi"] \\
B' \arrow[r, equals] & B' \arrow[r]                           & B                 
\end{tikzcd}
  \]
  As a corollary, $\Dir_{\ulcorner/D} \hookrightarrow \Dir_{D}$ is a
  subtopos inclusion.
\end{thm}
\begin{proof}
We have displayed both the action of the left adjoint and its unit --- the
universal map into the pullback. The
counit is always an isomorphism since pullbacks are unique up to unique
isomorphism.

That this is lex follows quickly from the fact that taking pullbacks commutes
with taking (finite) limits. 
\end{proof}

\printbibliography

\end{document}